\documentclass[11pt,reqno]{amsart}
\usepackage{amsfonts,amssymb,amsmath}
\usepackage{enumitem}
\usepackage{ytableau}
\usepackage{nicematrix}
\usepackage{tikz}
\usepackage{xcolor}
\usepackage{mathtools}
\usepackage{multirow}
\usepackage{float}

\definecolor{fgreen}{RGB}{44,144, 14}

\renewenvironment{proof}{{\bfseries Proof.}}{\qed}

\numberwithin{equation}{section} 
\newtheorem{theorem}{Theorem}[section] 
\newtheorem{proposition}[theorem]{Proposition} 
\newtheorem{corollary}[theorem]{Corollary} 
\newtheorem{lemma}[theorem]{Lemma} 

\theoremstyle{definition}
\newtheorem{definition}[theorem]{Definition} 
\newtheorem{remark}[theorem]{Remark} 

\usepackage[ 
colorlinks=true,citecolor=cyan,  urlcolor=blue,linkcolor=blue]{hyperref}

\def\R{\mathbb R}
\def\s{\mathbb S}
\def\C{\mathbb C}

\def\ib{\mathbf {i}}

\def\D{\mathbb D}

\def\N{\mathbb N}

\def\R{\mathbb R}

\def\p{\mathcal P}

\def\g{\mathcal G}

\def\M{{\mathrm M}}

\def\R{\mathbb {R}}
\def\C{\mathbb {C}}
\def\N{\mathbb {N}}

\def\ib{\mathbf {i}}

\def\g{\mathfrak {g}}

\def\p{\mathfrak {p}}

\def\s{\mathfrak {s}}

\def\lto{\longrightarrow}

\newcommand{\defref}[1]{Definition~\ref{#1}}
\newcommand{\secref}[1]{Section~\ref{#1}}
\newcommand{\thmref}[1]{Theorem~\ref{#1}}
\newcommand{\lemref}[1]{Lemma~\ref{#1}}

\newcommand{\propref}[1]{Proposition~\ref{#1}}

\begin{document}

\title[Adjoint reality in complex symplectic Lie algebra]{Strongly real adjoint orbits of complex symplectic Lie group}

 \author[ T. Lohan and C. Maity]{Tejbir Lohan and Chandan Maity}

\address{Indian Institute of Technology Kanpur, Kanpur-208016, Uttar Pradesh, India}

\email{tejbirlohan70@gmail.com
	}

\address{Indian Institute of Science Education and Research (IISER) Berhampur, 
	Berhampur-760010, Odisha, India}
\email{cmaity@iiserbpr.ac.in}

\makeatletter
\@namedef{subjclassname@2020}{\textup{2020} Mathematics Subject Classification}
\makeatother




\begin{abstract} We consider the adjoint action of the symplectic Lie group $\mathrm{Sp}(2n,\mathbb{C})$ on its Lie algebra $\mathfrak{sp}(2n,\mathbb{C})$. An element  $X \in \mathfrak{sp}(2n,\mathbb{C})$ is called $\mathrm{Ad}_{\mathrm{Sp}(2n,\mathbb{C})}$-real  if $ -X = \mathrm{Ad}(g)X$ for some $g \in \mathrm{Sp}(2n,\mathbb{C})$. Moreover, if  $ -X = \mathrm{Ad}(h)X $ for some involution  $h \in \mathrm{Sp}(2n,\mathbb{C})$, then $X \in \mathfrak{sp}(2n,\mathbb{C})$ is called strongly $\mathrm{Ad}_{\mathrm{Sp}(2n,\mathbb{C})}$-real. In this paper, we prove that for every element $X \in \mathfrak{sp}(2n,\mathbb{C})$, there exists a skew-involution $g \in  \mathrm{Sp}(2n,\mathbb{C})$ such that $-X =\mathrm{Ad}(g)X$. Furthermore, we classify the strongly $\mathrm{Ad}_{\mathrm{Sp}(2n,\mathbb{C})}$-real elements in $\mathfrak{sp}(2n,\mathbb{C})$. We also classify skew-Hamiltonian matrices that are similar to their negatives via a symplectic involution.
\end{abstract}

\maketitle 
\setcounter{tocdepth}{1}

\section{Introduction} 
Let $\mathrm{M}(n,\C)$ be the algebra of $n \times n$ matrices over $\C$, and ${\rm GL}(n,\C) $ be  the group of invertible elements in  ${\rm M}(n,\C)$.  
Consider the symplectic Lie group $ \mathrm{Sp}(2n,\C) := \{ g \in \mathrm{GL}(2n,\C) \mid g^T {\rm J}_{2n} g = {\rm J}_{2n} \}$ and its Lie algebra $\mathfrak{sp}(2n,\C)  := \{ X \in  \mathrm{M}(2n,\C) \mid X^T {\rm J}_{2n}  =- {\rm J}_{2n}X  \}$, where $
{\rm J}_{2n}  :=\begin{psmallmatrix}
	& \mathrm{I}_n\\
	- \mathrm{I}_n & 
\end{psmallmatrix}$ and $ \mathrm{I}_n$ denotes the $ n \times n $ identity matrix.
The elements of $ \mathrm{Sp}(2n,\C) $ and $\mathfrak{sp}(2n,\C)$ are  known in the literature as \textit{symplectic} and \textit{Hamiltonian} matrices, respectively.

Let $G$ be a group. An element of $ G$ is called \textit{reversible} or \textit{real} if it is conjugate to its inverse in $G$. 
An element of $ G$ is called  \textit{strongly reversible} or \textit{strongly real} if it is conjugate to its inverse by an involution (i.e., an element of order at most two) in $G$.   
It follows that an element of $ G$ is strongly reversible if and only if it can be expressed as a product of two involutions in $ G$. Moreover, every strongly reversible element in a group is reversible, but the converse is not always true. Such elements naturally appear in various areas, such as group theory, representation theory, geometry, complex analysis, functional equations, and classical dynamics; see \cite{ Wo,  TZ, BM,  O'Fa}. Thus, it has been a problem of broad interest    
to classify reversible and strongly reversible elements in a group; see \cite{OS} for an elaborate exposition of this theme.

Recently, in \cite{GM1}, the authors introduced the notion of adjoint reality  (an infinitesimal analog of classical reversibility) to a  Lie algebra using the natural adjoint action of a Lie group on the associated Lie algebra. Let $G$ be a Lie group with Lie algebra $\g$. Consider the  adjoint representation   ${\rm Ad} : G \lto \mathrm{GL}( \mathfrak{g} )$ of $G$ on $\g$. For $X\in \g$, the adjoint orbit of $X$ is the set $\{{\rm Ad}(g)X\mid g\in G \}$. 

\begin{definition}[cf.~{\cite[Definition 1.1]{GM1}}]
	An element $X\in \g$ is called $ {\rm Ad}_G$-real  if ${\rm Ad}(g)X  =-X$ for some $g\in G$.  An  element $X\in \g$  is called strongly $ {\rm Ad}_G$-real if $ {\rm Ad}(h)X  =-X  $ for some involution   $h\in G$.
\end{definition}
In the case of a linear Lie group $G$, the Ad-representation is given by conjugation, i.e., ${\rm Ad}(g)X = gXg^{-1}$.   
Note that if $X\in \g$ is  $ {\rm Ad}_G$-real (resp. strongly $ {\rm Ad}_G$-real), then $\exp X$ is reversible (resp. strongly reversible) in $G$. Using the notion of adjoint reality, the reversible and strongly reversible unipotent elements in classical simple Lie groups are classified in \cite{GM1}. 
This notion also plays a vital role in the investigation of reversibility in the general linear group $\mathrm{GL}(n,\D)$ and the affine group $\mathrm{GL}(n,\D) \ltimes \D^{n}$, where $\D =\R, \C$ or the division ring $\mathbb{H}$ of real quaternions; see  \cite{GLM2, GLM3}.

Let $G$ be a Lie group. The Lie algebra of $G$ is denoted by $\mathfrak{g}$ or ${\rm Lie}(G)$. A natural problem is to give a classification of the $ {\rm Ad}_G$-real and strongly $ {\rm Ad}_G$-real elements in $\mathfrak{g}$. We investigated this question in \cite{GLM1} for the special linear Lie algebra $\mathfrak{sl}(n,\mathbb{F})$ and  classify the $\mathrm{Ad}_{\mathrm{SL}(n,\mathbb{F})}$-real  and strongly $\mathrm{Ad}_{\mathrm{SL}(n,\mathbb{F})}$-real  orbits in  $\mathfrak{sl}(n,\mathbb{F})$, where $\mathbb{F}= \mathbb{C}$ or  $\mathbb{H}$. Recently, in \cite{GM2}, the authors investigated the adjoint reality of semisimple elements in complex simple classical Lie algebras. 
They also investigated  $\mathrm{Ad}_{\mathrm{Sp}(2n,\mathbb{C})}$-real  elements in $\mathfrak{sp}(2n,\mathbb{C})$ using the description of the centralizers of  nilpotent elements.  
In this article, we will revisit the adjoint reality problem in the complex symplectic Lie algebra $\mathfrak{sp}(2n,\mathbb{C})$.  
In our first result, we can say more about a reversing element that conjugates $X$ to $-X$.

\begin{theorem}\label{thm-real-skew-inv}
	For  every element $X \in \mathfrak{sp}(2n,\mathbb{C})$,  there exists a skew-involution $g$ (i.e., $g^2 =-{\rm I}_{2n}$) in  $\mathrm{Sp}(2n,\mathbb{C})$ such that $-X =gXg^{-1}$. Hence,  every element in $\mathrm{Lie(PSp}(2n,\C))$ is strongly $\mathrm{Ad}_{\mathrm{PSp}(2n,\mathbb{C})}$-real.  
\end{theorem}

Consequently, the following result follows immediately from  \thmref{thm-real-skew-inv}.

\begin{corollary} [cf.~{\cite[Theorem 4.2]{GM2}}] \label{cor-real-sp(2n,C)} 
	Every element of $\mathfrak{sp}(2n,\mathbb{C})$ is $\mathrm{Ad}_{\mathrm{Sp}(2n,\mathbb{C})}$-real.
\end{corollary}

Recall	 that  every element of $\mathrm{Sp}(2n,\mathbb{C})$ is conjugate to its inverse by a skew-involution in $\mathrm{Sp}(2n,\mathbb{C})$, and hence every element of symplectic Lie group $\mathrm{Sp}(2n,\mathbb{C})$ is reversible; see  \cite[Theorem 5.6]{JP}. 
It is worth mentioning that  \thmref{thm-real-skew-inv} can be thought of as a Lie algebra analog of \cite[Theorem 5.6]{JP}.

Our next  result  classifies the strongly $\mathrm{Ad}_{\mathrm{Sp}(2n,\mathbb{C})}$-real elements of $\mathfrak{sp}(2n,\mathbb{C})$.  We refer to \secref{subsec-canonical-form} for the definition of the Jordan block $ \mathrm{J}(\lambda, m)$ of size $m$ corresponding to the eigenvalue $\lambda \in \C$. 
\begin{theorem}\label{thm-main-str-real-sp(2n,C)}
	An element $X \in \mathfrak{sp}(2n,\mathbb{C})$ is strongly $\mathrm{Ad}_{\mathrm{Sp}(2n,\mathbb{C})}$-real  if and only if the Jordan blocks in the Jordan decomposition of  $X$ satisfy the following conditions:
	\begin{enumerate}
		\item \label{cond-1-main-thm}
		Every nilpotent Jordan block  $ \mathrm{J}(0, 2m)$ of even size has even multiplicity.
		
		\item\label{cond-2-main-thm} For every non-zero eigenvalue $\lambda$, the Jordan block $ \mathrm{J}(\lambda, k)$ has even multiplicity.
	\end{enumerate}
\end{theorem}


Our approach in this paper is based on certain canonical forms of elements in $\mathfrak{sp}(2n,\mathbb{C})$. 
We have suitably modified the canonical form given in \cite[Lemma 6]{Cr} for our purposes. The notion of the expanding sum of matrices (see \defref{def-expanding-sum}) and the structure of the reversing symmetry group for $X \in \mathfrak{sp}(2n,\mathbb{C})$ (see \secref{sec-rev-symm-gp}) play a vital role here.

Finally, in Section \ref{sec-skew-class}, we consider the set $\mathcal{SH}(2n, \C):= \{ X \in  \mathrm{M}(2n, \C) \mid X^T {\rm J}_{2n}  = {\rm J}_{2n}X  \}$ of \textit{skew-Hamiltonian} matrices in  ${\M}(2n, \C) $. Recall that  two matrices $A, B \in \mathrm{M}(m,\mathbb{C})$  are called {\it similar} if there exists a matrix $g\in {\rm GL}(m, \C)$ such that $gAg^{-1} =B$. 
Furthermore, when $A$ and $B$ are in $\mathrm{M}(2n,\mathbb{C})$, they are said to be  {\it symplectically similar} if there exists a symplectic matrix  $h \in \mathrm{Sp}(2n,\mathbb{C})$ such that $hAh^{-1} =B$. 
We prove the following result, which classifies the elements of $ \mathcal{SH}(2n, \C)$ that are symplectically similar to their own negatives.
\begin{theorem}\label{thm-class-skew-symp}
	An element $X \in \mathcal{SH}(2n,\C)$  is similar to $-X$ if and only if  $X$ is similar to $-X$  via a symplectic involution in $\mathrm{Sp}(2n,\mathbb{C})$.
\end{theorem} 

\textbf{Structure of the paper.}
In Section \ref{sec-prelim}, we recall some background and preliminary results. We investigate adjoint reality in $\mathfrak{sp}(2n,\mathbb{C})$ and prove  \thmref{thm-real-skew-inv}  in Section \ref{sec-adjoint-real}. Section \ref{sec-st-real} addresses the classification of strongly adjoint real elements in $\mathfrak{sp}(2n,\mathbb{C})$, and we prove our main result, Theorem \ref{thm-main-str-real-sp(2n,C)}. Finally,  we prove \thmref{thm-class-skew-symp}  in  Section \ref{sec-skew-class}.

\section{Preliminaries} \label{sec-prelim}
In this section, we fix some notation and recall some necessary background that will be used throughout this paper. 
For $ A\in {\M}(n, \C) $, let $A^T$ denote the {\it transpose} of the matrix $ A $.
First, we recall some primary results related to the symplectic Lie group $\mathrm{Sp}(2n,\mathbb{C})$ and its Lie algebra $\mathfrak{sp}(2n,\mathbb{C})$. 
Recall that the elements of $ \mathrm{Sp}(2n,\C) $ and $\mathfrak{sp}(2n,\C)$ are  known  as {symplectic} and Hamiltonian  matrices, respectively.   
The following observations immediately follow from the definitions and provide a criterion for checking whether a matrix is symplectic or Hamiltonian.
\begin{enumerate}[label = {({\bf  Ob.\arabic*})}]
	\item An element $g:=\begin{psmallmatrix}
		g_{1}	& g_{2}\\
		g_{3} & g_{4} 
	\end{psmallmatrix}\in \mathrm{Sp}(2n,\mathbb{C})$  if and only if  $g_{1}g_{4}^{T} -g_{2}g_{3} ^{T}= \mathrm{I}_{n}$ and both $g_{1}g_{2}^{T}$ and $g_{3}g_{4}^{T}$ are symmetric matrices.
	
	\item  An element 
	$A:=\begin{psmallmatrix}
		A_{1}	& A_{2}\\
		A_{3} & A_{4} 
	\end{psmallmatrix}\in \s\p(2n, \C)$  if and only if $A_{1} = -A_{4}^{T}$ and both $A_2$ and $A_3$ are symmetric matrices.
\end{enumerate} 

In particular,  for any  $g \in \mathrm{GL}(n,\mathbb{C})$,
$\begin{psmallmatrix}
	& g\\
	-({g}^{T})^{-1} & 
\end{psmallmatrix}$ and $\begin{psmallmatrix}
	g & \\
	& ({g}^{T})^{-1}
\end{psmallmatrix}$ are symplectic matrices in $\mathrm{Sp}(2n,\mathbb{C})$. 
The following result will be used to understand a suitable canonical form of the elements in $ \mathfrak{sp}(2n,\mathbb{C})$.
\begin{lemma}[cf.~{\cite[Corollary 22]{HM}}] 
	\label{lem-equi-symp-sim}
	Let $A$ and $B$ be both either symplectic or Hamiltonian matrices. Then $A$ and $B$ are similar if and only if $A$ and $B$ are symplectically similar.
\end{lemma}

Let $P \oplus Q:=\begin{psmallmatrix}
	P	& \\
	& Q
\end{psmallmatrix}  \in \mathrm{M}(m+n,\C)$ denote the direct sum of  the matrices $P \in \mathrm{M}(m,\C)$ and  $Q \in \mathrm{M}(n,\C)$. In the following definition, we recall the notion of the expanding sum of matrices; see \cite[page 385]{Cr}.
\begin{definition}[{Expanding sum of matrices}] \label{def-expanding-sum}
	Let $A=\begin{psmallmatrix}
		A_{1}	& A_{2}\\
		A_{3} & A_{4} 
	\end{psmallmatrix}$ and $B=\begin{psmallmatrix}
		B_{1}	& B_{2}\\
		B_{3} & B_{4} 
	\end{psmallmatrix}$, where $A_{i} \in \mathrm{M}(m,\C)$ and  $B_{i} \in \mathrm{M}(n,\C)$  for all $1\leq i \leq 4$. Then the \textit{expanding sum} of matrices $A$ and $B$ is defined as follows
	\begin{equation}\label{eq-expanding-sum}
		A \boxplus B :=  \begin{pmatrix}
			A_{1} \oplus B_{1}	& A_{2} \oplus B_{2}\\
			A_{3} \oplus B_{3} & A_{4} \oplus B_{4} 
		\end{pmatrix}.
	\end{equation}
\end{definition}

Observe that $\mathrm{I}_{2m}  \boxplus \mathrm{I}_{2n}  = \mathrm{I}_{2m +2n}$ and  $\mathrm{J}_{2m}   \boxplus \mathrm{J}_{2n} = \mathrm{J}_{2m +2n}$. For $A \in \mathrm{M}(n,\C)$, the  \textit{spectrum} $\sigma (A)$  denotes the set of eigenvalues of $A$, and the {\it centralizer} $\mathcal{Z}(A)$ of $A$  is defined as follows
\begin{equation}\label{eq-def-centralizer}
	\mathcal{Z}(A):= \{ B \in     \mathrm{M}(n,\C) \mid AB =BA   \}.
\end{equation}
Let $G\subset \mathrm{M}(n,\C) $ be a group, and $A\in G$. Then the centralizer $\mathcal{Z}_G(A)$ of $A$ in $G$ is defined as $\mathcal{Z}_G(A): =\mathcal{Z}(A)\, \cap\, G$.
Next, we recall some useful properties of the expanding sum of matrices; see  \cite{Cr, CP}. 
\begin{lemma}[{\cite[Lemma 4]{CP}}] \label{lem-prop-exp-sum}
	Let $A \in \mathrm{M}(m,\C)$ and  $B \in \mathrm{M}(n,\C)$.  Then the following statements hold.
	\begin{enumerate}
		\item The matrix $A \boxplus B$ is symplectic (resp. Hamiltonian)  if and only if both the matrices  $A$ and $B$ are symplectic (resp. Hamiltonian).
		\item The matrix $A \oplus B$ is similar to $A \boxplus B$ and $B \boxplus A$.
		
		\item $(A \boxplus B)^{T} = A^{T} \boxplus B^{T}$ and $(A \boxplus B)^{-1} = A^{-1} \boxplus B^{-1}$.
		
		\item Let $C \in \mathrm{M}(m,\C)$ and  $D \in \mathrm{M}(n,\C)$. Then $(A \boxplus B)(	C \boxplus D) = 	AC \boxplus BD$.
		\item Let $\sigma (A) \cap \sigma (B) = \emptyset$ and $f \in \mathcal{Z} (A \boxplus B)$. Then $f = f_{1} \boxplus f_{2}$, where $f_{1} \in \mathcal{Z} (A)$ and $f_{2} \in \mathcal{Z} (B)$.
	\end{enumerate}
\end{lemma}

If $A$ and $B$ are Hamiltonian matrices,  then using  \lemref{lem-equi-symp-sim} and \lemref{lem-prop-exp-sum}, it follows that $A \boxplus B$ is symplectically similar to  $B \boxplus A$. Similarly,
if $A$ and $B$ are symplectic involutions (resp. skew-involutions), then $A \boxplus B$ is a symplectic involution (resp. skew-involution).

\subsection{Canonical forms of matrices in  $ \mathfrak{sp}(2n,\mathbb{C})$  under symplectic similarity}\label{subsec-canonical-form}

Let $ \mathrm{J}(\lambda, m)$  denote the \textit{Jordan block} of size $m$ corresponding to eigenvalue $ \lambda  \in \C$, and it is defined as a square matrix of order $m$
with $ \lambda $ on the diagonal entries, $1$ on all of the super-diagonal entries, and $0$ elsewhere. We will refer to a block diagonal matrix in  $\mathrm{M}(n,\C)$ where each block is a Jordan block as  \textit{Jordan form}. Recall that every matrix in  $\mathrm{M}(n,\C)$ is similar (or conjugate) to a Jordan form, which is unique up to a permutation of Jordan blocks. The Jordan canonical form of symplectic and Hamiltonian matrices are studied in literature; see \cite{LMX}, \cite[Theorem 4, Theorem 5]{CMP}.

Since adjoint reality is invariant under conjugation, it is sufficient to work with suitable symplectic similar canonical forms of matrices in  $ \mathfrak{sp}(2n,\mathbb{C})$.
In the following lemma, we recall a canonical form of Hamiltonian matrices in $ \mathfrak{sp}(2n,\mathbb{C})$  under \textit{symplectic similarity}, called the \textit{symplectic Jordan form} of  Hamiltonian matrices; see \cite{Cr}.
\begin{lemma}[{\cite[Lemma 6]{Cr}}]\label{lem-canonical-form-Cruz}
	Each Hamiltonian matrix is symplectically similar to the expanding sum of matrices of the form
	\begin{equation}
		\mathrm{J}(\lambda, k) \oplus - \mathrm{J}(\lambda, k)^{T}  \quad (\lambda \in \C), \hbox{ and }  \begin{pmatrix}
			\mathrm{J}(0, l) & E_{ll}\\
			& - \mathrm{J}(0, l)^{T}
		\end{pmatrix},
	\end{equation}
	where $E_{ll} \in \mathrm{M}(l, \C)$ such that $(l, l)^{th}$ entry of $E_{ll}$ is one and all others entries are zero.
\end{lemma}

For $l \in \N$, define $\Delta_{2l}:= \begin{psmallmatrix}
	\mathrm{J}(0, l) & E_{ll}\\
	& - \mathrm{J}(0, l)^{T}
\end{psmallmatrix} $, and  $\Lambda_{2l} := \begin{psmallmatrix}
	\mathrm{J}(0, l) & \mathrm{I}_{l} \\
	& - \mathrm{J}(0, l)^{T}
\end{psmallmatrix} $. Note that $\Delta_{2l}$ and $\Lambda_{2l} $ both are Hamiltonian matrices   in $\mathfrak{sp}(2l,\C)$ similar to  $\mathrm{J}(0, 2l)$. In view of \lemref{lem-equi-symp-sim}, $\Delta_{2l}$ is symplectically similar to $\Lambda_{2l} $ for each $l \in \N$.  The next result follows from  \lemref{lem-canonical-form-Cruz}.

\begin{proposition}\label{prop-modified-canonical-form}
	Each Hamiltonian matrix in $ \mathfrak{sp}(2n,\mathbb{C})$ is symplectically similar to the expanding sum of matrices of the form
	\begin{equation}
		\mathrm{J}(\lambda, k) \oplus - \mathrm{J}(\lambda, k)^{T}  \quad (\lambda \in \C), \hbox{ and }  \begin{pmatrix}
			\mathrm{J}(0, l) & \mathrm{I}_{l} \\
			& - \mathrm{J}(0, l)^{T}
		\end{pmatrix}.
	\end{equation}
\end{proposition}

In this paper, we will work with the canonical form of elements in $ \mathfrak{sp}(2n,\mathbb{C})$ given in \propref{prop-modified-canonical-form}.

\subsection{Reversing symmetry group for  $ X \in \mathfrak{sp}(2n,\mathbb{C})$}\label{sec-rev-symm-gp}

Following the classical notion (see \cite{BR}, \cite[Section 2.1.4]{OS}), in this set-up,  we define the  \textit{reversing symmetry group} or \textit{extended centralizer} as follows. 
For an element $ X \in \mathfrak{sp}(2n,\mathbb{C})$, the {\it reverser} set is defined as 
$$ \mathcal{R}_{\mathrm{Sp}(2n,\mathbb{C})}(X) := \{ g \in \mathrm{Sp}(2n,\mathbb{C}) \mid gXg^{-1} = -X \}.$$
Define the \textit{reversing symmetry group}  $\mathcal{E}_{\mathrm{Sp}(2n,\mathbb{C})}(X) := \mathcal{Z}_{\mathrm{Sp}(2n,\mathbb{C})}(X) \cup \mathcal{R}_{\mathrm{Sp}(2n,\mathbb{C})}(X)$, where the centralizer  $\mathcal{Z}_{\mathrm{Sp}(2n,\mathbb{C})}(X)$ is  defined in \eqref{eq-def-centralizer}.
The set $\mathcal{R}_{\mathrm{Sp}(2n,\mathbb{C})}(X)$ of reversers (or reversing elements) for an $\mathrm{Ad}_{\mathrm{Sp}(2n,\mathbb{C})}$-real  element $X$ is a right coset of the centralizer $\mathcal{Z}_{\mathrm{Sp}(2n,\mathbb{C})}(X)$ of $X$. Thus, the \textit{reversing symmetry group}  $\mathcal{E}_{\mathrm{Sp}(2n,\mathbb{C})}(X)$ is a subgroup of $\mathrm{Sp}(2n,\mathbb{C})$ in which $\mathcal{Z}_{\mathrm{Sp}(2n,\mathbb{C})}(X)$ has index at most $2$.
Therefore, to find the reversing symmetry group $\mathcal{E}_{\mathrm{Sp}(2n,\mathbb{C})}(X)$ of an $\mathrm{Ad}_{\mathrm{Sp}(2n,\mathbb{C})}$-real element $ X \in \mathfrak{sp}(2n,\mathbb{C})$, it is enough to specify one reverser for symplectic Jordan form of $X$ that is not in the centralizer.  In \secref{sec-adjoint-real}, we will provide an explicit reverser for certain symplectic Jordan forms in $\mathfrak{sp}(2n,\mathbb{C})$.

\subsection{Preliminary results} We will recall some necessary well-known results in this subsection. The strongly $\mathrm{Ad}_{\mathrm{Sp}(2n,\mathbb{C})}$-real nilpotent  and $\mathrm{Ad}_{\mathrm{Sp}(2n,\mathbb{C})}$-real semisimple elements   in $ \mathfrak{sp}(2n,\mathbb{C})$   are classified in  \cite{GM1} and \cite{GM2}, respectively. 

\begin{lemma}[cf.~{\cite[Theorem 4.9]{GM1}}]\label{lem-str-real-nil-symp-alg}
A nilpotent element $X \in \mathfrak{sp}(2n,\mathbb{C})$ is strongly $\mathrm{Ad}_{\mathrm{Sp}(2n,\mathbb{C})}$-real  if and only if  every nilpotent Jordan block $ \mathrm{J}(0, 2m)$ of even size in the Jordan decomposition of $X$ has even multiplicity.
\end{lemma}

The following result characterizes the strongly reversible elements in $\mathrm{Sp}(2n,\mathbb{C})$. 
\begin{lemma}[cf.~{\cite[Theorem 8]{Cr}}]\label{lem-str-rev-symp-grp}
An element $g \in \mathrm{Sp}(2n,\mathbb{C})$ is strongly reversible in $\mathrm{Sp}(2n,\mathbb{C})$ if and only if for every (non-zero) eigenvalue $\lambda \in \mathbb{C}$, the Jordan block $\mathrm{J}(\lambda, k)$ in the Jordan decomposition of $g$ has even multiplicity.
\end{lemma}

In the next remark, we will fill up a gap in the proof of  \cite[Theorem 8]{Cr}. 
\begin{remark}\label{rem-st-rev-Symp-correction} 
In the proof of \cite[Theorem 8]{Cr},  an involution $H_2$ was constructed to claim that  $P_\lambda$ is strongly reversible in $\mathrm{Sp}(4k,\mathbb{C})$,  where $P_{\lambda} = A \oplus A^{-T} $ such that $A =  \mathrm{J}(\lambda, k) \oplus  (\mathrm{J}(\lambda, k)$,  $\lambda \neq \pm 1$.
We observe that the  involution $H_2$ does not conjugate $ P_{\lambda}$ to $P_{\lambda}^{-1}$, i.e.,  $H_2P_{\lambda}H_2 \not= P^{-1}_{\lambda}$. 
Nevertheless, this issue can be rectified using   \lemref{lem-str-real-pair-non-nil}.  
To see this, write $ \lambda = e^{\mu} $ for some non-zero  $\mu \in \C$. Define $X = P \oplus -P^{T} \in \mathfrak{sp}(4k,\mathbb{C})$, where $P = \mathrm{J}(\mu, k) \oplus \mathrm{J}(\mu, k)$.  Then $\exp(X) $ is symplectically similar to $P_\lambda$.  Using \lemref{lem-str-real-pair-non-nil}, we get that $X$ is strongly $\mathrm{Ad}_{\mathrm{Sp}(4k,\mathbb{C})}$-real. Hence,  $P_\lambda$ is strongly reversible in $\mathrm{Sp}(4k,\mathbb{C})$.  
	\qed
\end{remark}

\section{Adjoint real elements in $\mathfrak{sp}(2n,\mathbb{C})$}\label{sec-adjoint-real}
In this section, we will construct a reversing skew-involution for certain symplectic Jordan forms in $\mathfrak{sp}(2n,\mathbb{C})$.  First, we state a
well-known basic result without proof.
\begin{lemma}  \label{lem-basic-inv}
	Let  $\sigma, \tau \in \mathrm{GL}(n,\C)$ such that  $ \sigma= \mathrm{diag}(1, -1, 1, \dots, (-1)^{n-1})_{n \times n}$ and 
	\begin{equation}
		\tau=[x_{i,j}]_{n \times n}, \hbox{ where }	 x_{i,j} = \begin{cases}
			1 & \text{if $i+j =n+1$},\\
			0 & \text{otherwise}.
		\end{cases}
	\end{equation}
	Then the following statements hold.
	\begin{enumerate}
		
		\item $\sigma^2= \mathrm{I}_{n}$ and  $\sigma \mathrm{J}(0, n) = - \mathrm{J}(0, n)   \sigma$. 
		\item $\tau ^2 = \mathrm{I}_{n}$ and $\tau  (\mathrm{J}(\lambda, n)) = (\mathrm{J}(\lambda, n))^{T}\tau $ for all $\lambda \in \C$.
	\end{enumerate}
\end{lemma}

Next, we derive several useful facts from \lemref{lem-basic-inv} which will be used in proving \thmref{thm-real-skew-inv}.

\begin{lemma}\label{lem-real-nil-skew-construct}
	Let $X:=  \begin{psmallmatrix}
		\mathrm{J}(0,n) & \mathrm{I}_{n} \\
		& - \mathrm{J}(0,n)^{T}
	\end{psmallmatrix}$ be the symplectic Jordan form in $\mathfrak{sp}(2n,\mathbb{C})$. Then there exists a skew-involution $g\in \mathrm{Sp}(2n,\mathbb{C})$ such that $gXg^{-1}=-X$.
\end{lemma}
\begin{proof}
	Consider $g=  \begin{psmallmatrix}
		\sigma \ib & \\
		& - \sigma \ib
	\end{psmallmatrix}$, where $\sigma $ is an involution in $\mathrm{GL}(n,\C)$ as defined in \lemref{lem-basic-inv}. Then we get that  $g \in \mathrm{Sp}(2n,\mathbb{C})$ such that $ g^2= -\mathrm{I}_{2n}$. Observe  that $gX =-Xg$ if and only if $  \begin{psmallmatrix}
		\sigma \ib & \\
		& - \sigma \ib
	\end{psmallmatrix} \begin{psmallmatrix}
		\mathrm{J}(0,n) & \mathrm{I}_{n} \\
		& - \mathrm{J}(0,n)^{T}
	\end{psmallmatrix} =   \begin{psmallmatrix}
		-	\mathrm{J}(0,n) & -\mathrm{I}_{n} \\
		&  \mathrm{J}(0,n)^{T}
	\end{psmallmatrix}  \begin{psmallmatrix}
		\sigma \ib & \\
		& - \sigma \ib
	\end{psmallmatrix}$. This implies that 
	\begin{equation*}
		gX =-Xg	\Longleftrightarrow 	\begin{psmallmatrix}
			\sigma	\mathrm{J}(0,n) \ib & 	\sigma \ib \\
			& 	\sigma \mathrm{J}(0,n)^{T} \ib
		\end{psmallmatrix} =   \begin{psmallmatrix}
			-	\mathrm{J}(0,n) 	\sigma \ib & 	\sigma \ib\\
			&  -\mathrm{J}(0,n)^{T} 	\sigma \ib
		\end{psmallmatrix}  \Longleftrightarrow \sigma	\mathrm{J}(0,n) = -	\mathrm{J}(0,n)	\sigma.
	\end{equation*}
	The proof now follows from \lemref{lem-basic-inv}.
\end{proof}

\begin{lemma}\label{lem-real-non-nil-skew-construct}
	Let $X:= \mathrm{J}(\lambda, n) \oplus - (\mathrm{J}(\lambda, n))^{T}$ be the symplectic Jordan form in $\mathfrak{sp}(2n,\mathbb{C})$, where $\lambda \in \C$. Then there exists a skew-involution $g\in \mathrm{Sp}(2n,\mathbb{C})$ such that $gXg^{-1}=-X$.
\end{lemma}
\begin{proof}
	Consider $g=  \begin{psmallmatrix}
		& \tau\\
		-	\tau	& 
	\end{psmallmatrix}$, where $\tau \in \mathrm{GL}(n,\C)$ is an involution  as defined in \lemref{lem-basic-inv}. Then  $g \in \mathrm{Sp}(2n,\mathbb{C})$ and $ g^2= -\mathrm{I}_{2n}$. Note that $gX =-Xg$ if and only if $ \begin{psmallmatrix}
		& \tau\\
		-	\tau	& 
	\end{psmallmatrix} \begin{psmallmatrix}
		\mathrm{J}(\lambda, n) &  \\
		& -  \mathrm{J}(\lambda, n)^{T}
	\end{psmallmatrix} =   \begin{psmallmatrix}
		-	\mathrm{J}(\lambda, n) &  \\
		&  \mathrm{J}(\lambda, n)^{T}
	\end{psmallmatrix} \begin{psmallmatrix}
		& \tau\\
		-	\tau	& 
	\end{psmallmatrix}$. This implies that
	\begin{equation*}
		gX =-Xg	\Longleftrightarrow 	\begin{psmallmatrix}
			& 	-\tau  \mathrm{J}(\lambda, n)^{T}\\
			- \tau	 \mathrm{J}(\lambda, n)	& 	
		\end{psmallmatrix} =  	\begin{psmallmatrix}
			& 	- \mathrm{J}(\lambda, n) \tau \\
			- 	 \mathrm{J}(\lambda, n)^{T} \tau	& 	
		\end{psmallmatrix}  \Longleftrightarrow  \tau \mathrm{J}(\lambda, n) = 	\mathrm{J}(\lambda, n)^{T} \tau.
	\end{equation*}
	The proof now follows from \lemref{lem-basic-inv}.
\end{proof}

\subsection{Proof of \thmref{thm-real-skew-inv}} In view of \lemref{lem-prop-exp-sum}, the expanding sum of two symplectic skew-involutions is a symplectic skew-involution. The proof of  \thmref{thm-real-skew-inv} now follows from \propref{prop-modified-canonical-form}, \lemref{lem-real-nil-skew-construct}, and \lemref{lem-real-non-nil-skew-construct}.
\qed

\section{Strongly  adjoint real elements in $\mathfrak{sp}(2n,\mathbb{C})$}
\label{sec-st-real}
In this section, we will prove  \thmref{thm-main-str-real-sp(2n,C)}, which classifies the strongly $\mathrm{Ad}_{\mathrm{Sp}(2n,\mathbb{C})}$-real elements in $\mathfrak{sp}(2n,\mathbb{C})$.  The following result uses the structure of the reversing symmetry group for $X \in \mathfrak{sp}(2n,\mathbb{C})$, as introduced in \secref{sec-rev-symm-gp}. We refer to \defref{def-expanding-sum} for the notion of the expanding sum of two square matrices.
\begin{lemma}\label{lem-factor-str-real}
Let $X := X_{1} \boxplus  X_{2} \in \mathfrak{sp}(2n,\mathbb{C})$, where $X_{1} \in \mathfrak{sp}(2k,\mathbb{C})$ and $X_{2} \in \mathfrak{sp}(2n-2k,\mathbb{C})$ such that $k \in \N \cup \{0\}$ and $\sigma( X_{1}) \cap \sigma( X_{2}) =  \emptyset$. Then $X$ is strongly $\mathrm{Ad}_{\mathrm{Sp}(2n,\mathbb{C})}$-real element if and only if  $X_{1}$  is strongly $\mathrm{Ad}_{\mathrm{Sp}(2k,\mathbb{C})}$-real and    $ X_{2}$  is  strongly $\mathrm{Ad}_{\mathrm{Sp}(2n-2k,\mathbb{C})}$-real element, respectively.
\end{lemma}
\begin{proof} Suppose that $0<k<n$; otherwise we are done.
Since	$X$ is strongly $\mathrm{Ad}_{\mathrm{Sp}(2n,\mathbb{C})}$-real element, there exists $g \in \mathrm{Sp}(2n, \C)$ such that 
	$g^2 = \mathrm{I}_{2n}$ and $gXg^{-1} =-X$. In view of \propref{prop-modified-canonical-form}, \lemref{lem-real-nil-skew-construct} and \lemref{lem-real-non-nil-skew-construct}, we can construct  $h_1 \in \mathrm{Sp}(2k, \C)$ and  $h_2 \in \mathrm{Sp}(2n-2k, \C)$  such that $h_1X_1h_1^{-1}=-X_1$ and $h_2X_2 h_2^{-1}=-X_2$. Set $h :=h_1 \boxplus h_2$.  Then \lemref{lem-prop-exp-sum} implies that $h \in \mathrm{Sp}(2n, \C)$ such that $hXh^{-1} =-X$.

Since the set of reversers of $X  \in \mathfrak{sp}(2n,\mathbb{C})$ is a right coset of the centralizer of $X$, we have
	$$g = fh,$$
	where	$f \in \mathrm{M}(2n, \C)$ such that $fX=Xf$; see Section \ref{sec-rev-symm-gp}.
	Using \lemref{lem-prop-exp-sum}, we get that 
	$$f =f_1 \boxplus f_2,$$
	$f_1 \in \mathrm{M}(2k, \C)$ and  $f_2 \in \mathrm{M}(2n-2k, \C)$  such that $f_1X_1=X_1f_1$ and $f_2X_2=X_2f_2$, respectively. Therefore, we have
	$$g = fh= g_1 \boxplus g_2,$$
	where $g_1 =f_1h_1 \in \mathrm{Sp}(2k, \C)$ and  $g_2 =f_2h_2 \in \mathrm{Sp}(2n-2k, \C)$. Moreover, the equations $g^2 = \mathrm{I}_{2n}$ and $gX g^{-1}=-X$  imply that
	$$g_{1}^{2} = \mathrm{I}_{2k}, \, g_1X_1 g_{1}^{-1}=-X_1, \hbox{ and } g_{2}^{2} = \mathrm{I}_{2n-2k}, \, g_2X_2 g_{2}^{-1}=-X_2.
	$$
	
Conversely, recall that the expanding sum of two symplectic involutions is a symplectic involution. The proof now follows from \lemref{lem-prop-exp-sum}.
\end{proof}

\subsection{Strong  reality of certain symplectic Jordan forms in $\mathfrak{sp}(2n,\mathbb{C})$.} In this subsection, we investigate certain strongly $\mathrm{Ad}_{\mathrm{Sp}(2n,\mathbb{C})}$-real symplectic Jordan forms in $\mathfrak{sp}(2n,\mathbb{C})$.

\begin{lemma}\label{lem-str-real-odd-nil}
Let $X:= \mathrm{J}(0, n) \oplus - (\mathrm{J}(0, n))^{T}$ be the symplectic Jordan form in $\mathfrak{sp}(2n,\mathbb{C})$. Then $X$ is strongly $\mathrm{Ad}_{\mathrm{Sp}(2n,\mathbb{C})}$-real.
\end{lemma}
\begin{proof}
Consider $g=  \begin{psmallmatrix}
		\sigma & \\
		& \sigma
	\end{psmallmatrix}$, where $\sigma $ is an involution in $\mathrm{GL}(n,\C)$ as defined in \lemref{lem-basic-inv}. Then using a similar line of arguments as used in the proof of \lemref{lem-real-nil-skew-construct}, we get that $g \in \mathrm{Sp}(2n,\mathbb{C})$ such that $ g^2= \mathrm{I}_{2n}$  and $gXg^{-1}=-X$. This proves the lemma.
\end{proof}

\begin{corollary}\label{cor-str-real-pair-even-nil}
Let $X :=Y \boxplus Y \in \mathfrak{sp}(4n,\mathbb{C}) $, where $Y=  \begin{psmallmatrix}
		\mathrm{J}(0,n) & \mathrm{I}_{n} \\
		& - \mathrm{J}(0,n)^{T}
	\end{psmallmatrix}$ be the symplectic Jordan form in $\mathfrak{sp}(2n,\mathbb{C})$. Then $X$ is strongly $\mathrm{Ad}_{\mathrm{Sp}(4n,\mathbb{C})}$-real.
\end{corollary}
\begin{proof}
In view of \lemref{lem-equi-symp-sim}, 	$X$ is symplectically similar to $\mathrm{J}(0, 2n) \oplus - (\mathrm{J}(0, 2n))^{T}$ in $\mathfrak{sp}(4n,\mathbb{C}) $. Hence, \lemref{lem-str-real-odd-nil} implies that $X$ is strongly $\mathrm{Ad}_{\mathrm{Sp}(4n,\mathbb{C})}$-real.
\end{proof}

\begin{lemma}\label{lem-str-real-pair-non-nil} Let $\lambda$ be a non-zero complex number.  Consider $X :=Y \boxplus Y \in \mathfrak{sp}(4n,\mathbb{C}) $, where $Y= \mathrm{J}(\lambda, n) \oplus - (\mathrm{J}(\lambda, n))^{T}$ is the symplectic Jordan form in $\mathfrak{sp}(2n,\mathbb{C})$. Then $X$ is strongly $\mathrm{Ad}_{\mathrm{Sp}(4n,\mathbb{C})}$-real.
\end{lemma}
\begin{proof}
	Note that $X = P \oplus -P^{T}$, where $P = \mathrm{J}(\lambda, n) \oplus \mathrm{J}(\lambda, n)$. Consider $g = \begin{psmallmatrix}
		& h\\
		-	h	& 
	\end{psmallmatrix}$ such that $h= \begin{psmallmatrix}
		& \tau \\
		-	\tau 	& 
	\end{psmallmatrix}$,  where $\tau $ is an involution in $\mathrm{GL}(n,\C)$ as defined in \lemref{lem-basic-inv}. Since $h^2 = -\mathrm{I}_{2n}$ and $h^{T} =-h$, we have $g^2 = \mathrm{I}_{4n}$ and  $g \in \mathrm{Sp}(4n,\mathbb{C})$. Note that
	\begin{equation*}
		gX =-Xg  \Longleftrightarrow hP= P^{T}h  \Longleftrightarrow \tau \mathrm{J}(\lambda, n) = \mathrm{J}(\lambda, n)^{T} \tau.
	\end{equation*}
	The proof now follows from \lemref{lem-basic-inv}. 
\end{proof}

\subsection{Proof of \thmref{thm-main-str-real-sp(2n,C)}}  In view of \propref{prop-modified-canonical-form}, up to symplectic similarity, we can assume that $X$ has  the following form:
$$X = X_{1} \boxplus  X_{2} \in \mathfrak{sp}(2n,\mathbb{C}),$$ where $X_{1} \in \mathfrak{sp}(2k,\mathbb{C})$ and $X_{2} \in \mathfrak{sp}(2n-2k,\mathbb{C})$ such that $k \in \N \cup \{0\}$,  $\sigma( X_{1}) \cap \sigma( X_{2}) =  \emptyset$ and  $0$ is only  eigenvalue of $X_{1}$ (i.e., $X_1$ is a nilpotent or zero matrix).

 Let $X \in \mathfrak{sp}(2n,\mathbb{C})$  be strongly $\mathrm{Ad}_{\mathrm{Sp}(2n,\mathbb{C})}$-real. Then \lemref{lem-factor-str-real} implies that $X_{1}$ and $ X_{2}$ are strongly $\mathrm{Ad}_{\mathrm{Sp}(2k,\mathbb{C})}$-real and  strongly $\mathrm{Ad}_{\mathrm{Sp}(2n-2k,\mathbb{C})}$-real element, respectively. Suppose that $X$ is a non-zero matrix such that $k \neq n$; otherwise the proof follows from   \lemref{lem-str-real-nil-symp-alg}. Now, if $k=0$, then $X =X_2$. Furthermore, for $0<k<n$, if $X_1$ is a (non-zero) nilpotent element in $\mathfrak{sp}(2k,\mathbb{C})$, then  the  condition \eqref{cond-1-main-thm}  of \thmref{thm-main-str-real-sp(2n,C)} holds using  \lemref{lem-str-real-nil-symp-alg}. Note that $X_2 \in \mathfrak{sp}(2n-2k,\mathbb{C})$ is a strongly $\mathrm{Ad}_{\mathrm{Sp}(2n-2k,\mathbb{C})}$-real element and has only non-zero eigenvalues. 

Suppose that $X_2 \in \mathfrak{sp}(2n-2k,\mathbb{C})$ has a non-zero eigenvalue $\lambda \in \C$ such that there exists a Jordan block $\mathrm{J}(\lambda, s)$ of odd multiplicity $t\in \N$ in the Jordan decomposition of $X_2$, where $s \in \N$ and $0 \leq k <n$. In view of \propref{prop-modified-canonical-form}, up to symplectic similarity, we can assume that 
$$X_2 = X_{11} \boxplus  X_{12} \in \mathfrak{sp}(2n-2k,\mathbb{C}),$$ where $X_{11} \in \mathfrak{sp}(2m,\mathbb{C})$ and $X_{12} \in \mathfrak{sp}(2n-2k-2m,\mathbb{C})$ such that $m \in \N$, $k \in \N \cup \{0\}$, $ k < n$, $\sigma( X_{11}) \cap \sigma( X_{12}) =  \emptyset$ and $X_{11}$ has only $\lambda$ and $-\lambda$ as an eigenvalues. In view of  \lemref{lem-factor-str-real}, we get that $X_{11}$ is strongly $\mathrm{Ad}_{\mathrm{Sp}(2m,\mathbb{C})}$-real.   Therefore, $\exp(X_{11}) \in \mathrm{Sp}(2m,\mathbb{C})$ is strongly reversible in $\mathrm{Sp}(2m,\mathbb{C})$. However, the Jordan decomposition of $X_{11}$ has the Jordan block $\mathrm{J}(\exp(\lambda), s)$ with odd multiplicity $t$, and this contradicts with  \lemref{lem-str-rev-symp-grp}. Therefore, every Jordan block $\mathrm{J}(\lambda, s)$ in the Jordan decomposition of $X_2 \in \mathfrak{sp}(2n-2k,\mathbb{C}) $ has even multiplicity, where $0 \leq k<n$.
Hence,  if $X \in \mathfrak{sp}(2n,\mathbb{C})$  is strongly $\mathrm{Ad}_{\mathrm{Sp}(2n,\mathbb{C})}$-real, then the conditions \eqref{cond-1-main-thm}  and \eqref{cond-2-main-thm}  of \thmref{thm-main-str-real-sp(2n,C)} hold true. 

Conversely, let both  the conditions \eqref{cond-1-main-thm} and \eqref{cond-2-main-thm}  of \thmref{thm-main-str-real-sp(2n,C)} hold true. In view of \propref{prop-modified-canonical-form}, up to symplectic similarity, we can assume that $X$ can be written as an expanding sum of matrices of the form
\begin{equation*}
	\mathrm{J}(0, s) \oplus - (\mathrm{J}(0, s))^{T}, \hbox{ and } (\mathrm{J}(\lambda, t) \oplus - (\mathrm{J}(\lambda, t))^{T}) \boxplus (\mathrm{J}(\lambda, t) \oplus - (\mathrm{J}(\lambda, t))^{T}),
\end{equation*}
where $\lambda$ is a non-zero complex number and $s,t \in \N$. Recall that the expanding sum of two symplectic involutions is a symplectic involution; see  \lemref{lem-prop-exp-sum}. 
Therefore,  using  \lemref{lem-str-real-odd-nil}
and \lemref{lem-str-real-pair-non-nil}, we can construct a suitable involution $g$ in $\mathrm{Sp}(2n,\mathbb{C})$ such that $gXg^{-1} = -X$. Hence, $X$ is strongly $\mathrm{Ad}_{\mathrm{Sp}(2n,\mathbb{C})}$-real.  This completes the proof. 
\qed

\section{Skew-Hamiltonian matrices that are similar to their own negatives.} \label{sec-skew-class}
Recall that $\mathcal{SH}(2n,\C):= \{ X \in  \mathrm{M}(2n,\C) \mid X^T {\rm J}_{2n}  = {\rm J}_{2n}X  \}$ and the elements of $\mathcal{SH}(2n,\C)$ are known as skew-Hamiltonian matrices in  ${\M}(2n, \C) $.  An element 
$A:=\begin{psmallmatrix}
	A_{1}	& A_{2}\\
	A_{3} & A_{4} 
\end{psmallmatrix}\in \mathcal{SH}(2n,\C)$  if and only if $A_{1} = A_{4}^{T}$ and both $A_2$ and $A_3$ are skew-symmetric matrices. In particular, $\begin{psmallmatrix}
	P	& \\
	& P^{T}
\end{psmallmatrix} \in \mathcal{SH}(2n,\C)$ for all $P \in  \mathrm{M}(n,\C)$. Note that for  $A \in \mathrm{M}(m,\C)$ and  $B \in \mathrm{M}(n,\C)$,   $A \boxplus B$ is skew-Hamiltonian if and only if both the matrices  $A$ and $B$ are skew-Hamiltonian matrices, cf. \lemref{lem-prop-exp-sum}. In the following lemma, we recall a canonical form of skew-Hamiltonian matrices in $\mathcal{SH}(2n,\C)$ under symplectic similarity.
\begin{lemma}[{\cite[Lemma 6]{Cr}}]\label{lem-canonical-form-skew-Hamil-Cruz}
Each skew-Hamiltonian matrix is symplectically similar to the expanding sum of matrices of the form
	\begin{equation}
		\mathrm{J}(\lambda, k) \oplus  \mathrm{J}(\lambda, k)^{T}  
	\end{equation}
	where $\lambda \in \C$.
\end{lemma}

The following result classifies the skew-Hamiltonian matrices, which are similar to their own negatives.
\begin{lemma}\label{lem-skew-sim-neg-class}
An element  $X \in \mathcal{SH}(2n,\C)$ is similar to $-X$ if and only if $X$ is symplectically similar to the expanding sum of matrices of the form
	\begin{equation}\label{eq-form-skew-sim-neg}
		P_{0}, \quad
		\hbox{and} \quad Q_{\lambda}\oplus Q_{\lambda}^{T}
	\end{equation}
	where $P_{0}	:=
	\mathrm{J}(0, m) \oplus
	\mathrm{J}(0, m)^{T} \in \mathcal{SH}(2m,\C)$, $Q_{\lambda} := \mathrm{J}(\lambda, k)  \oplus - \mathrm{J}(\lambda, k) \in \mathrm{GL}(2k,\C)$ and $\lambda \in \C$ is non-zero.
\end{lemma}
\begin{proof}
Note that two matrices $X, Y \in \mathcal{SH}(2n,\C)$  are similar if and only if $X$ and $Y$ are symplectically similar; see {\cite[Corollary 22]{HM}}.  Therefore,  $X \in \mathcal{SH}(2n,\C)$ is similar to $-X$ if and only if $X $ is symplectically similar to $-X$.  Using \eqref{eq-expanding-sum}, we have  
	\begin{equation*}
		(	\mathrm{J}(\lambda, k)  \oplus  \mathrm{J}(\lambda, k)^{T}  )  \boxplus 	(-\mathrm{J}(\lambda, k)  \oplus  -\mathrm{J}(\lambda, k)^{T} ) = Q_{\lambda}\oplus Q_{\lambda}^{T},
	\end{equation*}
	for all non-zero $\lambda \in \C$. The proof now follows from \lemref{lem-canonical-form-skew-Hamil-Cruz}.
\end{proof}

\subsection{Proof of \thmref{thm-class-skew-symp}} Suppose that $X \in \mathcal{SH}(2n,\C)$ is similar to $-X$. In view of \lemref{lem-skew-sim-neg-class}, up to symplectic similarity, we can assume that $X$ can be written as an expanding sum of matrices of the form given in \eqref{eq-form-skew-sim-neg}. Consider symplectic involutions $g_{0} \in \mathrm{Sp}(2m, \C)$ and $g_{\lambda} \in \mathrm{Sp}(4k, \C)$  such that
\begin{equation}\label{eq-conj-skew-form}
	g_{0} :=  \begin{pmatrix}
		\sigma & \\
		&  \sigma 
	\end{pmatrix}, \hbox{ and } g_{\lambda}:=  \begin{pmatrix}
		h & \\
		& h
	\end{pmatrix} \quad (\lambda \in \C \setminus \{0\}), 
\end{equation}
where $\sigma \in \mathrm{GL}(m,\C)$ is an involution as defined in \lemref{lem-basic-inv} and $h=  \begin{pmatrix}
	&  \mathrm{I}_{k}\\
	\mathrm{I}_{k} 	& 
\end{pmatrix} \in \mathrm{GL}(2k,\C)$. Then we have,
\begin{equation*}
	g_{0} P_{0}  = - P_{0} g_{0} \hbox{ and } g_{\lambda} (Q_{\lambda}\oplus Q_{\lambda}^{T})  = - (Q_{\lambda}\oplus Q_{\lambda}^{T}) g_{\lambda}, 
\end{equation*}
where $P_{0}  \in \mathcal{SH}(2m,\C) $ and $Q_{\lambda}\oplus Q_{\lambda}^{T}  \in \mathcal{SH}(4k,\C) $ are as defined in \eqref{eq-form-skew-sim-neg}.  In view of \lemref{lem-prop-exp-sum}, the expanding sum of two symplectic involutions is a symplectic involution. Therefore,  using  $g_{0}$ and $g_{\lambda}$ given in \eqref{eq-conj-skew-form}, we can construct a suitable involution $g \in \mathrm{Sp}(2n,\mathbb{C})$ such that $gXg^{-1} = -X$.  Hence, the proof of the theorem follows.
\qed

\subsection*{Acknowledgment}  The authors would like to thank K. Gongopadhyay for his comments on the first draft of this paper. It is a great pleasure to thank the referee(s) for carefully reading the manuscript and providing many valuable comments. Lohan acknowledges the financial support from the IIT Kanpur Postdoctoral Fellowship, while  Maity is partially supported by the Seed Grant IISERBPR/RD/OO/2024/23.

Parts of this work were completed while the authors were visiting IISER Mohali and the Institute for Mathematical Sciences at the National University of Singapore in May-June 2024. The authors thank these institutes for their hospitality and support.

\end{document}